\documentclass{article}
\usepackage{amsmath, amscd, amssymb, latexsym, epsfig, color, amsthm, mathtools, tikz-cd, array, adjustbox, bbm, tikz, enumitem, relsize}

\makeatletter
\newtheorem*{rep@theorem}{\rep@title}
\newcommand{\newreptheorem}[2]{%
	\newenvironment{rep#1}[1]{%
		\def\rep@title{#2 \ref{##1}}%
		\begin{rep@theorem}}%
		{\end{rep@theorem}}}
\makeatother

\numberwithin{equation}{section}
\newreptheorem{theorem}{Theorem}
\newtheorem{theorem}{Theorem}[section]
\newtheorem{proposition}[theorem]{Proposition}
\newtheorem{question}[theorem]{Question}
\newtheorem{corollary}[theorem]{Corollary}
\newtheorem{conjecture}[theorem]{Conjecture}
\newtheorem{lemma}[theorem]{Lemma}

\theoremstyle{definition}

\newtheorem{remark}[theorem]{Remark}

\DeclareMathOperator\lk{\mathrm{lk}}
\DeclareMathOperator\st{\mathrm{st}}
\DeclareMathOperator\cost{\mathrm{cost}}
\DeclareMathOperator{\Hilb}{\mathrm{Hilb}}

\newcommand{\field}{\mathbbm{k}}
\newcommand{\NN}{{\mathbb N}}

\newcommand{\QQ}{{\mathbb Q}}
\newcommand{\CC}{{\mathbb C}}

\newcommand{\ZZ}{{\mathbb Z}}

\newcommand{\mideal}{\ensuremath{\mathfrak{m}}}

\newcommand{\Hom}{\ensuremath{\mathrm{Hom}}\hspace{1pt}}
\newcommand{\Ext}{\ensuremath{\mathrm{Ext}}\hspace{1pt}}

\title{Stanley--Reisner rings of simplicial complexes with a free action by an abelian group}
\author{Connor Sawaske\\
	\small Department of Mathematics\\[-0.8ex]
	\small University of Washington\\[-0.8ex]
	\small Seattle, WA 98195-4350, USA\\[-0.8ex]
	\small \texttt{sawaske@math.washington.edu}
}

\begin{document}
	\maketitle
	
	\begin{abstract}
		We consider simplicial complexes admitting a free action by an abelian group. Specifically, we establish a refinement of the classic result of Hochster  describing the local cohomology modules of the associated Stanley--Reisner ring, demonstrating that the topological structure of the free action extends to the algebraic setting. If the complex in question is also Buchsbaum, this new description allows for a specialization of Schenzel's calculation of the Hilbert series of some of the ring's Artinian reductions. In further application, we generalize to the Buchsbaum case the results of Stanley and Adin that provide a lower bound on the $h$-vector of a Cohen-Macaulay complex admitting a free action by a cyclic group of prime order.
	\end{abstract}

\section{Introduction}

Since the 1970's, an extensive dictionary translating the topological and combinatorial properties of a simplicial complex $\Delta$ into algebraic properties of the associated Stanley--Reisner ring $\field[\Delta]$ has been constructed. For instance, the local cohomology modules $H_\mideal^i(\field[\Delta])$ of $\field[\Delta]$ have a beautiful interpretation due to Hochster (and later Gr\"abe) as topological invariants of $\Delta$ (see \cite[Theorem 2]{Grabe}, \cite{Reisner}, or \cite[Section II.4]{St-96}). The following formulation of their results is the starting point of this paper (here $\cost_\Delta \sigma$ denotes the contrastar of a face $\sigma$ of $\Delta$, while $H^{i-1}(\Delta, \cost_\Delta \sigma)$ is the relative simplicial cohomology of the pair $(\Delta, \cost_\Delta \sigma)$ computed with coefficients in $\field$ and $H_\mideal^i(\field[\Delta])$ is the local cohomology module of $\field[\Delta]$):
\begin{theorem}\label{Reisner}Let $\Delta$ be a simplicial complex on vertex set $\{1, \ldots, n\}$ and let $\field$ be a field. Then
	\[
	H_\mideal^i(\field[\Delta])_{-j}\cong \bigoplus_{\mathclap{\substack{(u_1, \ldots, u_n)\in \NN^n,\,\,\, \sum u_k = j \\ \{k: u_k >0\}\in \Delta}}}H^{i-1}(\Delta, \cost_\Delta \{i: u_i>0\})
	\]
	as vector spaces over $\field$. 
\end{theorem}

On the other hand, the effect on the Stanley--Reisner ring of a simplicial complex being endowed with a group action has also been studied in some depth (see \cite[Section III.8]{St-96} and \cite{AdinThesis}). Our primary goal is to bring these two topics together by studying the additional structure on $H_\mideal^i(\field[\Delta])$ that appears when $\Delta$ admits a group action. As it turns out, the topological invariants introduced by the group action dictate a more detailed description of $H_\mideal^i(\field[\Delta])$, providing our main result.

Our secondary goal is to apply this new decomposition of $H_\mideal^i(\field[\Delta])$ to extend two separate classes of previously-existing results describing the $h$-vectors (perhaps the most widely-recognized and studied combinatorial invariants) of certain simplicial complexes. One of these classes of results examines Buchsbaum simplicial complexes: using Theorem \ref{Reisner}, Schenzel was able to calculate the Hilbert series of the quotient of $\field[\Delta]$ by a linear system of parameters in the case that $\Delta$ is Buchsbaum. The culmination of this line of study was the following theorem appearing in \cite{Schenzel}, providing an algebraic interpretation of the $h$-vector of $\Delta$.
\begin{theorem}[Schenzel]\label{Schenzel}Let $\Delta$ be a $(d-1)$-dimensional Buchsbaum simplicial complex, and let $\Theta$ be a linear system of parameters for $\Delta$. Then
	\[
	\dim_\field \left(\field[\Delta]/\Theta\field[\Delta]\right)_i= h_i(\Delta)+{d\choose i}\sum_{j=0}^{i-1}(-1)^{i-j-1}\beta_{j-1}(\Delta),
	\]
where $\beta_{j-1}(\Delta)$ denotes the reduced simplicial Betti number of $\Delta$ computed over $\field$.
\end{theorem}
Since any dimension must be non-negative, the theorem provides lower bounds for the entries of $h(\Delta)$ in terms of some topological invariants of $\Delta$. More recently, these bounds have been further lowered by the use of socles (\cite{NS-Socle}) and the sigma module $\Sigma(\Theta; \field[\Delta])$ (originally introduced by Goto in \cite{Goto}, and whose definition is deferred to Section \ref{sect:preliminaries}).
\begin{theorem}[Murai--Novik--Yoshida]\label{NS-Socle}Let $\Delta$ be a $(d-1)$-dimensional Buchsbaum simplicial complex and let $\Theta$ be a linear system of parameters for $\Delta$. Then
	\[
	\dim_\field \left(\field[\Delta]/\Sigma(\Theta; \field[\Delta])\right)_i= h_i(\Delta)+{d\choose i}\sum_{j=0}^{i}(-1)^{i-j-1}\beta_{j-1}(\Delta).
	\]
\end{theorem}
Furthermore, the symmetry appearing in the $h$-vector described by Gr\"{a}be in \cite{GrabeDS} was reformulated by Novik and Murai in \cite[Proposition 1.1]{FaceNumbers}, and using Theorem \ref{NS-Socle} it may be described as follows:
\begin{theorem}[Murai--Novik--Yoshida]\label{Duality}Let $\Delta$ be a $(d-1)$-dimensional Buchsbaum simplicial complex and let $\Theta$ be a linear system of parameters for $\Delta$. Then
	\[
	\dim_\field \left(\field[\Delta]/\Sigma(\Theta; \field[\Delta])\right)_i=\dim_\field \left(\field[\Delta]/\Sigma(\Theta; \field[\Delta])\right)_{d-i}
	\]
	for $i=0, \ldots, d$.
\end{theorem}

The second class of results that we will enlarge deals with the $h$-vectors of complexes admitting specific group actions by $\ZZ/p\ZZ$ in the Cohen--Macaulay case. In particular, the following two theorems (\cite[Theorem 3.2]{St-87} and \cite[Theorem 3.3]{AdinThesis}, respectively) provide impressive bounds on the $h$-vector of such a complex and are ripe for extensions to more general complexes:
\begin{theorem}[Stanley]\label{StanleyCMVeryFree}
	Let $\Delta$ be a $(d-1)$-dimensional Cohen-Macaulay (over $\CC$) simplicial complex admitting a very free action by $\ZZ/p\ZZ$ with $p$ a prime. Then $
	h_i(\Delta)\ge {d\choose i}$ if $i$ is even and $h_i(\Delta)\ge (p-1){d\choose i}$ if $i$ is odd.
\end{theorem}

\begin{theorem}[Adin]\label{AdinCMFree}
	Let $\Delta$ be a $(d-1)$-dimensional Cohen-Macaulay (over $\CC$) simplicial complex with a free action of $\ZZ/p\ZZ$ with $p$ a prime such that $d$ is divisible by $p-1$. Then
	\[
	\sum_{i=0}^dh_i(\Delta)\lambda^i\ge (1+\lambda+\cdots +\lambda^{p-1})^{d/(p-1)},
	\]
	where the inequality holds coefficient-wise.
\end{theorem}
In summary, our new results are the following (as much notation must be introduced before explicit statements may be provided, we present here a brief overview):
\begin{itemize}\itemsep0em
	\item \textit{A new version of Hochster and Gr\"{a}be's theorem:} We examine how a group action leads to a special decomposition of $H_\mideal^i(\field[\Delta])$, and we demonstrate how the isomorphism of Theorem \ref{Reisner} respects this decomposition (Theorem \ref{ReisnerAnalog}). Later, for the $\ZZ/p\ZZ$ case, we encounter a piece of the local cohomology of $\field[\Delta]$ that corresponds to the singular cohomology of the quotient $|\Delta|/G$.
	\item \textit{Applications to $h$-vectors:} The dimension calculations of Theorems \ref{Schenzel} and \ref{NS-Socle} are refined and specialized to the case of a Buchsbaum complex admitting a free action by a cyclic group of prime order (Theorems \ref{SchenzelAnalog} and \ref{SigmaModuleThm}). Using this refinement, Theorems \ref{StanleyCMVeryFree} and \ref{AdinCMFree} are generalized to the setting of Buchsbaum complexes; in some cases, the expressions are identical (Section \ref{Inequalities}). Lastly, in Theorems \ref{DSTheorem} and \ref{DSTheorem2}, we exhibit a symmetry in the dimensions of finely graded pieces of some Artinian reductions of Stanley--Reisner rings of certain orientable homology manifolds admitting a free group action by a cyclic group of prime order in the flavor of Theorem \ref{Duality}.
\end{itemize}
The paper is organized as follows. In Section \ref{sect:preliminaries}, we review notation and some classical results. In Section \ref{ReisnerSection} we study the local cohomology modules $H_\mideal^i(\field[\Delta])$ and provide our main new result. Section \ref{Applications} is devoted to calculating a finely-graded Hilbert series of an Artinian reduction of the Stanley--Reisner ring $\field[\Delta]$, providing many applications of our main theorem. We close with comments and questions in Section \ref{comments}.


\section{Preliminaries} \label{sect:preliminaries}


\subsection{Combinatorics and topology}

A \textbf{simplicial complex} $\Delta$ on vertex set $[n]:=\{1, \ldots, n\}$ is a collection of subsets of $[n]$ that is closed under inclusion. The elements of $\Delta$ are called \textbf{faces}, and the maximal faces (with respect to inclusion) are \textbf{facets}. The \textbf{dimension} of a face $\sigma$ is defined by $\dim \sigma :=|\sigma|-1$, and the \textbf{dimension} of $\Delta$ is defined by $\dim \Delta := \max\{\dim \sigma: \sigma\in \Delta\}$. The faces of dimension zero are called \textbf{vertices}. If all facets of $\Delta$ have the same dimension, then we say that $\Delta$ is \textbf{pure}.

Given a face $\sigma$ of $\Delta$, we define the \textbf{star} of $\sigma$ in $\Delta$ by
\[
\st_\Delta \sigma :=\{\tau\in \Delta: \sigma\cup \tau\in \Delta\}.
\]
and the \textbf{contrastar} of $\sigma$ in $\Delta$ by
\[
\cost_\Delta \sigma := \{\tau\in \Delta: \sigma\not\subset\tau\}.
\]
Similarly, the $\textbf{link}$ of $\sigma$ in $\Delta$ is
\[
\lk_\Delta \sigma := \{\tau\in \Delta: \sigma\cup \tau \in \Delta, \sigma\cap \tau = \emptyset\}.
\]
We define the $f$-\textbf{vector} $f(\Delta)$ of a $(d-1)$-dimensional simplicial complex $\Delta$ by $f(\Delta)=(f_{-1}(\Delta), f_0(\Delta), \ldots, f_{d-1}(\Delta))$, where $f_i(\Delta)$ is the number of $i$-dimensional faces of $\Delta$. The $h$-\textbf{vector} $h(\Delta)$ is then defined by $h(\Delta)=(h_0(\Delta), h_1(\Delta), \ldots, h_d(\Delta))$, where
\[
h_i(\Delta)=\sum_{j=0}^i(-1)^{i-j}{d-j\choose i-j}f_{j-1}(\Delta).
\]

Let $\field$ be a field, and let $\tilde{H}^i(\Delta)$ be the $i$-th reduced simplicial cohomology group of $\Delta$ with coefficients in $\field$. If $\Gamma$ is a subcomplex of $\Delta$, then $H^i(\Delta, \Gamma)$ is the $i$-th relative cohomology group of the pair $(\Delta, \Gamma)$ with coefficients in $\field$. In the case that $\Gamma=\{\emptyset\}$, this is the same as the reduced cohomology group $\tilde{H}^i(\Delta)$. Denote the $i$\textbf{-th (reduced) Betti number} of $\Delta$ over $\field$ by
\[
\beta_i(\Delta):=\dim_\field \widetilde{H}^i(\Delta).
\]

We call a complex $\Delta$ \textbf{Cohen-Macaulay} if $\beta_i(\lk_\Delta \sigma)=0$ for all faces $\sigma$ and all $i<\dim \lk_\Delta \sigma$. Similarly, we call a complex \textbf{Buchsbaum} if $\Delta$ is pure and $\beta_i(\lk_\Delta \sigma)=0$ for all faces $\sigma\not=\emptyset$ and all $i<\dim \lk_\Delta \sigma$. If the link of each non-empty face $\sigma$ of $\Delta$ has the homology of a $(d-|\sigma|-1)$-sphere, then we say that $\Delta$ is a $\field$\textbf{-homology manifold}; furthermore, if $\beta_{\dim\Delta}(\Delta)$ is equal to the number of connected components of $\Delta$, then we say that $\Delta$ is \textbf{orientable}.

Now let $G$ be a finite group and suppose $G$ acts on a simplicial complex $\Delta$ (in particular, each $g\in G$ acts as a simplicial isomorphism on $\Delta$). We say that $G$ acts \textbf{freely} if $g\cdot \sigma \not= \sigma$ for all faces $\sigma\not=\emptyset$ and all non-identity elements $g\in G$. Similarly, $G$ acts \textbf{very freely} if for every face $\sigma$ and every non-identity element $g\in G$,
\[
(\st_\Delta \sigma) \cap (\st_\Delta (g\cdot\sigma)) = \{\emptyset\}.
\]
 In the case that $G=\ZZ/2\ZZ$, free and very free actions are equivalent and we call $\Delta$ \textbf{centrally-symmetric}. When extending a free group action by $G$ to the geometric realization $|\Delta|$ of $\Delta$, the action obtained is a covering space action (see \cite[Section 1.3]{Hatcher}).

Lastly, given a vector $U=(i_1, \ldots, i_n)\in \NN^n$, let $s(U)=\{j:i_j>0\}$ be the support of $U$ and let $|U|=\sum_{j=1}^n i_j$ be the $L^1$-norm of $U$. Denote
\begin{equation}\label{TjDefinition}
T(\Delta)_j=\{U\in \NN^n:\text{$s(U)\in \Delta$ and $|U|=j$}\}.
\end{equation}
\subsection{Modules and group actions}
Assume from now on that the field $\field$ is an extension of $\CC$. If $G$ is a finite group and $N$ is a $\field[G]$-module, then $N$ admits a decomposition into isotypic components according to the action of $G$. That is,
\[
N=\bigoplus_{\chi} N^\chi
\]
where the sum is taken over all irreducible characters $\chi$ of $G$ and
\[
N^\chi := \{ n\in N: g\cdot n = \chi(g) n\text{ for all $g\in G$}\}.
\]
Of vital importance will be the consideration of each graded piece
\[
\bigoplus_{U\in T(\Delta)_j}H^{i-1}(\Delta, \cost_\Delta s(U))
\]
of $H_\mideal^i(\field[\Delta])$ as a $\field[G]$-module when $\Delta$ admits a group action by $G$ (each direct sum of reduced simplicial cohomology groups inherits a natural $G$-action from $\Delta$). Our Theorem \ref{ReisnerAnalog} will consider the isotypic components 
\[
\left[\bigoplus_{U\in T(\Delta)_j}H^i(\Delta, \cost_\Delta s(U))\right]^{\mathlarger{\chi}}.
\]
In the $j=0$ case, we set
\begin{equation}\label{bettinotation}
\beta_i(\Delta)^\chi:=\dim_\field \widetilde{H}^i(\Delta)^\chi.
\end{equation}
These refined Betti number will be one of the main invariants considered in Sections 4 and 5 of this paper.

For the rest of this section, let $G=\ZZ/p\ZZ$ for some prime $p$ and let $\zeta$ be a primitive $p$-th root of unity. In this case, the group of irreducible characters of $G$ is isomorphic to $G$ itself. We can fix a generator $g$ for $G$ and write
\[
N=\bigoplus_{j=0}^{p-1}N^j,
\]
where
\[
N^j:=\{n\in N: g\cdot n = \zeta^jn\}.
\]
As in our Definition (\ref{bettinotation}), once a generator $g$ is fixed we set
\begin{equation}\label{bettiJnotation}
\beta_i(\Delta)^j:=\dim_\field \widetilde{H}^i(\Delta)^j.
\end{equation}

Now let $A=\field[x_1, \ldots, x_n]$, let $M$ be a $\ZZ$-graded $A$-module, and let $\deg(m)$ denote the degree of an element $m\in M$, i.e., $M=\bigoplus_iM_i$ where $M_i=\{m\in M: \deg(m)=i\}$. If $G$ acts on $M$ and fixes $M_i$ for all $i$, then each $M_i$ can be thought of as a $\field[G]$-module. As above,
\[
M_i=\bigoplus_{j=0}^{p-1} M_i^j
\]
where
\[
M_i^j=\{m\in M:\deg(m)=i\text{ and }g\cdot m = \zeta^jm\}.
\]
This allows for a $(\ZZ\times G)$-grading for $M$, where $M_i^j$ is the component of $M$ consisting of all elements of degree $(i, g^j)$. Denote by $M[a, b]$ the ``shifted'' $A$-module whose component in degree $(i, g^j)$ is the $(i+a, g^{j+b})$-th component of $M$. We will only ever use these shifted modules in examining individual pieces in some short exact sequences of $\ZZ\times G$-graded vector spaces in Section \ref{SchenzelSection}; to this end, we do not consider $M[a, b]$ as having any specific $G$-action of its own. Thinking now of $M$ as a $(\ZZ\times G)$-graded vector space over $\field$, we can define the Hilbert series $\Hilb(M, \lambda, t)$ of $M$ by
\[
\Hilb(M, \lambda, t)=\sum_{(i, g^j)\in (\ZZ\times G)}(\dim_\field M_i^j)\lambda^it^j
\]
where $\lambda$ and $t$ are indeterminates with $t^p=1$.

If $M$ is of Krull dimension $d>0$, we call a system $\Theta = \theta_1, \theta_2, \ldots, \theta_d$ of homogeneous elements in $A$ a \textbf{homogeneous system of parameters} (or an \text{h.s.o.p.}) for $M$ if $M/\Theta M$ is a finite-dimensional vector space over $\field$. If each $\theta_i\in A_1$, then we call $\Theta$ a \textbf{linear system of parameters} (or an \text{l.s.o.p.}) for $M$. We say that $M$ is \textbf{Cohen-Macaulay} if every l.s.o.p. is a regular sequence on $M$, and $M$ is \textbf{Buchsbaum} if every l.s.o.p. satisfies
\[
(\theta_1, \ldots, \theta_{i-1})M:_M\theta_i = (\theta_1, \ldots, \theta_{i-1})M:_M\mideal
\]
for $i=1, \ldots, d$. Here $\mideal=(x_1, \ldots, x_n)$ is the irrelevant ideal of $A$. Given any h.s.o.p. $\Theta$ for $M$, we also have the notion of the sigma module $\Sigma(\Theta; M)$, defined by
\[
\Sigma(\Theta; M):=\Theta M + \sum_{i=0}^d\left((\theta_1, \ldots, \hat{\theta_i}, \ldots, \theta_d)M:_M\theta_i\right).
\]
Lastly, let $H_\mideal^i(M)$ denote the $i$-th local cohomology module of $M$ with respect to $\mideal$ (the construction and properties of these modules may be found in \cite{24Hours}).

\subsection{Stanley--Reisner rings}

Let $\Delta$ be a $(d-1)$-dimensional simplicial complex with vertex set $[n]$ and define $A:=\field[x_1, \ldots, x_n]$. Given $\sigma\subset [n]$, write $x_\sigma=\prod_{i\in \sigma}x_i$. The \textbf{Stanley--Reisner ideal} of $\Delta$ is the ideal $I_\Delta$ of $A$ defined by
\[
I_\Delta=(x_\sigma:\sigma\subset [n], \sigma\not\in\Delta).
\]
The \textbf{Stanley--Reisner ring} of $\Delta$ (over $\field$) is 
\[
\field[\Delta]:= A/I_\Delta.
\]
The geometric notion of Buchsbaumness is tied algebraically to the Stanley--Reisner ring through the following vital result (found in \cite{Schenzel}).
\begin{theorem}[Schenzel]A pure simplicial complex $\Delta$ is Buchsbaum over $\field$ if and only if $\field[\Delta]$ is a Buchsbaum $A$-module.
\end{theorem}

If $\Delta$ admits an action by the group $G=\ZZ/p\ZZ$, then this extends to an action on $\field[\Delta]$ and induces a $(\ZZ\times G)$ grading as detailed in the previous section. If this action is free, then for any $j$ and $i\ge 1$ we have
\[
\dim_\field \field[\Delta]_i^j=\frac{1}{p}\dim_\field \field[\Delta]_i.
\]
We also have $\dim_\field \field[\Delta]_0^0=1$ and $\dim_\field \field[\Delta]_0^j=0$ for $0<j<p$. As in \cite[Section 3]{St-87}, this implies the following expression for the $(\ZZ\times G)$-graded Hilbert series of $\field[\Delta]$.
\begin{theorem}\label{SRHilb}
	Let $\Delta$ be a $(d-1)$-dimensional simplicial complex admitting a free action by the group $\ZZ/p\ZZ$. Then
	\[
	\Hilb(\field[\Delta], \lambda, t)=1+\frac{1}{p}\left[\frac{\sum_{i=1}^d h_i(\Delta)\lambda^i}{(1-\lambda)^d}-1\right](1+t+\cdots+t^{p-1}).
	\]
\end{theorem}
In fact, if $\Delta$ is centrally-symmetric (so $p=2$) and Cohen-Macaulay, then a certain Hilbert series provides a strong inequality bounding the $h$-vector of $\Delta$; see \cite[Theorem III.8.1]{St-96}. If $\Theta$ is an l.s.o.p. for $\Delta$, then we denote by $\field(\Delta; \Theta)$ the quotient $\field[\Delta]/\Theta\field[\Delta]$.
\begin{theorem}[Stanley]Let $\Delta$ be a $(d-1)$-dimensional centrally-symmetric Cohen-Macaulay simplicial complex, and suppose $\Theta=(\theta_1, \ldots, \theta_d)$ is a l.s.o.p. for $\field[\Delta]$ in which $\theta_i\in\field[\Delta]_1^1$ for $i=1, \ldots, d$. Then
	\[
	\Hilb(\field(\Delta; \Theta), \lambda, t)=\frac{1}{2}\left[(1-t)(1+\lambda)^d+(1+t)\sum_{i=0}^dh_i(\Delta)\lambda^i\right].
	\]
\end{theorem}

\section{The main theorem: a refinement of Hochster}\label{ReisnerSection}
For the rest of the paper, fix $\Delta$ to be a $(d-1)$-dimensional simplicial complex on vertex set $[n]$. Let $\field$ be an extension of $\CC$, and let $A$ be the polynomial ring $\field[x_1, \ldots, x_n]$. Recall that $T(\Delta)_j$ is the set $\{U\in \NN^n: s(U)\in \Delta$ and $ |U|=j\}$.

\begin{theorem}\label{ReisnerAnalog}
	Let $\Delta$ be a simplicial complex with a free action by a finite abelian group $G$. Then the isomorphisms
	\[
	H_\mideal^i(\field[\Delta])_{-j}\cong \bigoplus_{U\in T(\Delta)_j}H^{i-1}(\Delta, \cost_\Delta s(U))
	\]
	of Theorem \ref{Reisner} induce vector space isomorphisms
	\begin{equation}\label{gradingEq}
	H_\mideal^i(\field[\Delta])_{-j}^\chi\cong \left[\bigoplus_{U\in T(\Delta)_j}H^{i-1}(\Delta, \cost_\Delta s(U))\right]^{\chi^{-1}}
	\end{equation}
	for each irreducible character $\chi$ of $G$.
\end{theorem}

\begin{remark}There are a multitude of proofs of Theorem \ref{Reisner}, such as \cite[Theorem 1]{Grabe} or \cite[Corollary 4.4]{canonicalMap}. Our proof of the refinement above will be a modification of that which appears in \cite{characterizations}.
\end{remark}

\begin{proof}
From the proof of \cite[Theorem 1]{characterizations}, we know already that
\[
H_\mideal^i(\field[\Delta])_{-j}=\Ext_{A}^i(A/\mideal_{j+1}, \field[\Delta])_{-j},
\]
where $\mideal_{j+1}$ is the ideal $(x_1^{j+1}, \ldots, x_n^{j+1})$. The $A$-module $\Ext_{A}^i(A/\mideal_{j+1}, \field[\Delta])_{-j}$ may be computed as the cohomology of a cochain complex:
\[
\Ext_{A}^i(A/\mideal_{j+1}, \field[\Delta])_{-j}= H^i(\Hom_A(K_\cdot, \field[\Delta])_{-j}),
\]
where $K_\cdot$ denotes the Koszul complex of $A$ with respect to the sequence $(x_1^{j+1}, \ldots, x_n^{j+1})$. Viewing the right-hand side as a $\field[G]$-module with the natural action inherited from $\field[\Delta]$ and writing $C^\cdot(\Delta, \Gamma)$ for the relative simplicial cochain complex of the pair $(\Delta, \Gamma)$ with coefficients in $\field$, it suffices to show that the isomorphism
\begin{equation}\label{ExtAndCochainComplex}
H^i(\Hom_A(K_\cdot, \field[\Delta])_{-j})\cong \bigoplus_{U\in T(\Delta)_j}H^{i-1}(C^\cdot(\Delta, \cost_\Delta s(U))).
\end{equation}
``inverts'' isotypic components as in the statement of the theorem.

First note that the cochain complex $C^\cdot(\Delta, \Gamma)$ has a standard basis written as $\{\widehat{\sigma}:\sigma\in(\Delta\smallsetminus\Gamma) \}$, where $\widehat{\sigma}(\tau)=1$ if $\sigma=\tau$ and $\widehat{\sigma}(\tau)=0$ otherwise. On the left-hand side of (\ref{ExtAndCochainComplex}), we will view each $K_t$ (for $0\le t\in \ZZ$) as the direct sum
\[
K_t = \bigoplus_{1\le i_1<i_2<\cdots<i_t\le n}A(x_{i_1}^{j+1}\wedge x_{i_2}^{j+1}\wedge\cdots\wedge x_{i_t}^{j+1}).
\]
Given some set $\sigma=\{i_1, \ldots, i_t\}\subset [n]$ with $i_1<i_2<\cdots<i_t$, define $\overline{x}_\sigma:=x_{i_1}^{j+1}\wedge\cdots\wedge x_{i_t}^{j+1}$ in $K_t$. Likewise, if $U=(i_1, \ldots, i_n)\in \NN^n$, define $x_U:=x_{1}^{i_1}\cdots x_{n}^{i_n}$ in $\field[\Delta]$. If $\sigma, \tau\subset [n]$ are such that $|\sigma|=|\tau|=t$ and $U\in \NN^n$ is such that $s(U)\subset \tau$, define
\[
f_{\sigma, \tau}^U(\overline{x}_\rho)=\left\{\begin{array}{cc} (x_\tau)^{j+1}/x_U & \rho=\sigma\\ 0 & \rho\not=\sigma.\end{array}\right.
\]
Then $\Hom_A(K_t, \field[\Delta])_{-j}$ has a vector subspace with basis given by
\[
D^t=\{f_{\sigma, \tau}^U: |\sigma|=|\tau|=t, \text{ $\tau\in \Delta$, $|U|=j$ and }s(U)\subset \tau \}.
\]
Given a face $\sigma\in\Delta$, denote by $\mathcal{O}(\sigma):=\{g\cdot \sigma: g\in G\}$, the orbit of $\sigma$ under the action of $G$. To each face $\sigma$ in $\Delta$ and each $U\in \NN^n$ with $|U|=j$ and $s(U)\subset \sigma$, assign the homomorphism
\[
f_\sigma^U:=\sum_{\tau\in \mathcal{O}(\sigma)} f_{\tau, \sigma}^U
\]
and note that $g\cdot f_\sigma^U=f_{g\cdot\sigma}^{g\cdot U}$ by the action of $G$ on $\field[\Delta]$. Now define a map
\[
\varphi_{{}_U}:C^{t-1}(\Delta, \cost_\Delta s(U))\to \Hom_A(K_t, \field[\Delta])_{-j}
\]
by
\[
\varphi_{{}_U}:\widehat{\sigma}+C^{t-1}(\cost_\Delta s(U))\mapsto f_\sigma^U
\]
and let $\varphi$ be the direct sum of maps $\bigoplus_{U\in T(\Delta)_j} \varphi_U$. Note that $\varphi(g\cdot c)=g^{-1}\cdot\varphi(c)$ for all $g\in G$.

Two important facts follow from Miyazaki and Reisner's proofs of Theorem \ref{Reisner}. From Miyazaki's proof, we know that the direct sum of maps
\[
\psi:\bigoplus_{U\in T(\Delta)_j}C^{t-1}(\Delta, \cost_\Delta s(U))\to \Hom_A(K_t, \field[\Delta])_{-j}
\]
defined componentwise by
\[
\psi(\widehat{\sigma}+C^{t-1}(\cost_\Delta s(U)))=f_{\sigma, \sigma}^U
\]
is a chain map and induces an isomorphism in homology. Reisner's proof (\cite[Theorem 2]{Reisner}) shows that the image of the difference $\varphi-\psi$ lies entirely in an acyclic component of $\Hom_A(K_\cdot, \field[\Delta])_{-j}$ (this occurs because the action is free, i.e., $g\cdot \sigma$ and $h\cdot \tau$ are not equal for any $g\not= h$ in $G$). Hence, while $\varphi$ may not itself be a chain map, it does induce the same isomorphism of vector spaces as $\psi$.

It remains to check that $\varphi$ maps
\[
\left[\bigoplus_{|U|=j} C^{t-1}(\Delta, \cost_\Delta s(U))\right]^\chi
\]
to $\Hom_A(K_t, \field[\Delta])_{-j}^{\chi^{-1}}$. Given $U\in T(\Delta)_j$ and $\sigma\not\in \cost_\Delta s(U)$, the span of
\begin{align*}
\widehat{\mathcal{O}}(\widehat{\sigma}):&=\{g\cdot (\widehat{\sigma}+C^{t-1}(\cost_\Delta s(U)):g\in G\}
\end{align*}
forms a $\field[G]$-submodule of
\[
\bigoplus_{U\in T(\Delta)_j} C^{t-1}(\Delta, \cost_\Delta s(U)).
\]
Let
\[
c=\sum_{g\in G}a_g\left[g\cdot (\widehat{\sigma}+C^{t-1}(\cost_\Delta s(U))\right]
\]
be a basis element for the isotypic component of the span of $\widehat{\mathcal{O}}(\sigma)$ corresponding to some character $\chi$. That is, given $h\in G$, we have $h\cdot c = \chi(h)c$. Then
\begin{align*}
h\cdot \sum_{g\in G}a_g\left[g\cdot (\widehat{\sigma}+C^{t-1}(\cost_\Delta s(U))\right]
&= \sum_{g\in G}a_g\left[hg\cdot (\widehat{\sigma}+C^{t-1}(\cost_\Delta s(U))\right]\\
&=\sum_{g\in G}a_{h^{-1}g}\left[g\cdot (\widehat{\sigma}+C^{t-1}(\cost_\Delta s(U))\right].
\end{align*}
Since the $(\widehat{\sigma}+C^{t-1}(\cost_\Delta s(U))$'s are linearly independent over $\field$, this shows that $a_{h^{-1}g}=\chi(h)a_g$.

We now only need to check that $h\cdot \varphi(c)=\chi^{-1}(h)\varphi(c)$. Note that
\[
\varphi(c)=\sum_{g\in G}a_g(g^{-1}\cdot f_\sigma^U),
\]
so
\[
h\cdot \varphi(c)=\sum_{g\in G}a_g(hg^{-1}\cdot f_{\sigma}^U)=\sum_{g\in G}a_{hg}(g^{-1}\cdot f_\sigma^U)
\]
Since $a_{hg}=\chi(h^{-1})a_g=\chi^{-1}(h)a_g$, this finishes the proof.
\end{proof}

\section{Applications}\label{Applications}

\subsection{Hilbert series of Artinian reductions and an analog of Schenzel's formula}\label{SchenzelSection}
From now on, $\Delta$ is fixed to be Buchsbaum and $G=\ZZ/p\ZZ$ for some prime $p$ with a fixed generator $g$. In a way similar to \cite[Section 8]{St-96}, if the action of $G$ is very free and $0\le \delta_i\le p-1$ for $i=1, \ldots, d$, then we can easily construct a l.s.o.p. $\Theta=\theta_1, \ldots, \theta_d$ for $\field[\Delta]$ in which $\theta_i\in A_1^{\delta_i}$ for $1\le i \le d$ as follows.

First, choose one face from each $G$-orbit of $\Delta$ and collect them into the set $\Delta_G$. Let $W$ be the set of vertices of $\Delta$ that are in $\Delta_G$, and choose functions $t_1, \ldots, t_d:W\to\field$ such that their restrictions to any subset of $W$ of size $d$ are linearly independent. Now extend $t_i$ to all of $[n]$ by setting
\[
t_i(g^k\cdot v)=\zeta^{-k\delta_i}t_i(v)
\]
for all $i$ and $k$, and let
\[
\theta_i=\sum_{v\in [n]}t_i(v)x_v.
\]
Then
\begin{align*}
g\cdot\theta_i &= \sum_{v\in [n]}t_i(v)x_{g\cdot v}=\sum_{v\in [n]}t_i(g^{-1}\cdot v)x_v=\sum_{v\in [n]}\zeta^{\delta_i}t_i(v)x_v,
\end{align*}
so $\theta_i\in\field[\Delta]_1^{\delta_i}$ for $i=1, \ldots, d$. Furthermore, since no facet of $\Delta$ contains $\{g^i\cdot v, g^j\cdot v\}$ for any $v$ with $i\not\equiv j \mod p$, the system $\Theta=(\theta_1, \ldots, \theta_d)$ forms an l.s.o.p. that is homogeneous with respect to the $(\ZZ\times G)$-grading by \cite[Lemma III.2.4(a)]{St-96}.

\begin{remark}\label{freeactionlsop}The construction above is included purely for the sake of concreteness in the case of a very free action. In fact, as Adin shows in his thesis (\cite{AdinThesis}), it is possible to construct an l.s.o.p. with the prescribed properties above even in the case of a free action. The construction involves changing the field $\field$ to a particular extension of $\CC$. However, as we are now turning our focus to certain Hilbert series (which are invariant under field extensions), all of the results that follow hold for any extension of $\CC$ and are stated with this understanding in mind.
\end{remark}

Now that the existence of l.s.o.p.'s of the form above has been established, we may use them to prove the following theorem (recall the refined Betti number notation of definition (\ref{bettiJnotation}) and that $\field(\Delta; \Theta)=\field[\Delta]/\Theta\field[\Delta]$ for an l.s.o.p. $\Theta$).

\begin{theorem}\label{SchenzelAnalog}Let $\Delta$ be a $(d-1)$-dimensional Buchsbaum simplicial complex admitting a free group action by $G=\ZZ/p\ZZ$, let $0\le m\le p-1$ be some fixed degree, and let $\Theta$ be a $G$-homogeneous l.s.o.p. for $\field[\Delta]$ such that $\theta_i\in\field[\Delta]_1^m$ for all $i$. Then the $(\ZZ\times G)$-graded Hilbert series of $\field(\Delta;\Theta)$ is given by
\begin{align*}
\Hilb(\field(\Delta; \Theta), \lambda, t)=&\sum_{i=0}^d\left[(-1)^i{d\choose i}t^{mi}+\left(\frac{1}{p}\sum_{k=0}^{p-1}t^k\right)\left(h_i(\Delta)+(-1)^{i+1}{d\choose i}\right)\right]\lambda^i\\
&+\sum_{i=0}^d{d\choose i}\lambda^i\sum_{j=0}^{i-1}(-1)^{i-j-1}\left(\sum_{k=0}^{p-1}t^k\beta_{j-1}(\Delta)^{mi-k} \right).
\end{align*}
\end{theorem}

\begin{remark}Note that by setting $t=1$, we can recover Schenzel's Theorem \ref{Schenzel}. Furthermore, we conclude that $h_i(\Delta) \equiv (-1)^i{d\choose i} \pmod p$ for all $i$.
\end{remark}

\begin{proof}
Suppose that $\Theta$ is an arbitrary l.s.o.p.\@ with $\theta_i\in\field[\Delta]_i^{\delta_i}$ for $i=1, \ldots, d$. For $s=1, \ldots, d$, denote
\[
\field_s[\Delta]:=\field[\Delta]/(\theta_1, \ldots, \theta_s)\field[\Delta].
\]
Then we have exact sequences of the form
\[
0\to Q_s\to \field_{s-1}[\Delta]\xrightarrow{\cdot\theta_s} \field_{s-1}[\Delta]\to \field_s[\Delta]\to 0
\]
where
\[
Q_s=0:_{\field_{s-1}[\Delta]}\theta_s=0:_{\field_{s-1}[\Delta]}\mideal = H_\mideal^0(\field_{s-1}[\Delta]),
\]
with the second equality following from \cite[Proposition I.1.10(i)']{StVo} and the third from the proof of \cite[Proposition I.2.2]{StVo}. On the level of Hilbert series, a standard argument yields the following equation:
\[
\Hilb(\field(\Delta; \Theta), \lambda, t)=\Hilb(\field[\Delta], \lambda, t)\prod_{i=1}^d(1-\lambda t^{\delta_i})+\sum_{s=1}^d\lambda t^{\delta_i}\Hilb(H_\mideal^0(\field_{s-1}[\Delta]), \lambda, t)\prod_{j=s+1}^d(1-\lambda t^{\delta_j}).
\]
In the case $\theta_i\in \field[\Delta]_i^m$ for $i=1, \ldots, d$, the equation above simplifies to
\[
\Hilb(\field(\Delta; \Theta), \lambda, t)=(1-\lambda t^m)^d\Hilb(\field[\Delta], \lambda, t)+\sum_{s=1}^d\lambda t^m(1-\lambda t^m)^{d-s}\Hilb(H_\mideal^0(\field_{s-1}[\Delta]), \lambda, t).
\]
Analyzing the first term yields the following, using Theorem \ref{SRHilb}:
\begin{align*}
(1-\lambda t^m)^d\Hilb(\field[\Delta], \lambda, t)&=(1-\lambda t^m)^d\left[1+\frac{1}{p}\left(\frac{\sum_{i=0}^d h_i(\Delta)\lambda^i}{(1-\lambda)^d}-1\right)(1+t+\cdots+t^{p-1})\right]\\
&=(1-\lambda t^m)^d+\frac{(1+t+\cdots+t^{p-1})}{p}\left[\sum_{i=0}^dh_i(\Delta)\lambda^i - (1-\lambda t^m)^d\right]\\
&=\sum_{i=0}^d\left[(-1)^i{d\choose i}t^{mi}+\frac{1}{p}\sum_{k=0}^{p-1}t^k\left(h_i(\Delta)+(-1)^{i+1}{d\choose i}\right)\right]\lambda^i.
\end{align*}
For the second term (recall the shifted module notation from Section \ref{sect:preliminaries}),
\[
H_\mideal^0(\field_{s-1}[\Delta])\cong\bigoplus_{i=0}^{s-1}\left(\bigoplus_{{s-1\choose i}}H_\mideal^i(\field[\Delta])[-i, -im]\right)
\]
as vector spaces by \cite[Proposition II.4.14']{StVo} (strictly speaking, the stated result is for $\ZZ$-graded $A$-modules, but the exact same proof works in the $(\ZZ\times G)$-graded case when only considering both sides as vector spaces).

By our Theorem \ref{ReisnerAnalog}, $H^i_\mideal(\field[\Delta])[-i, -im]^k$ is concentrated in $\ZZ$-degree $i$ and has dimension $\beta_{i-1}(\Delta)^{mi-k}$ in degree $(i, g^k)$. Hence,
\[
\Hilb(H_\mideal^0(\field_{s-1}[\Delta]), \lambda, t)=\sum_{i=0}^d{s-1\choose i}\lambda^i\left(\sum_{k=0}^{p-1}t^k\beta_{i-1}(\Delta)^{mi-k}\right),
\]
and so the $\lambda^it^k$ coefficient of $\lambda t^m\Hilb(H_\mideal^0(\field_{s-1}[\Delta]), \lambda, t)$ is ${s-1\choose i-1}\beta_{i-2}(\Delta)^{mi-k}$. Then the $\lambda^it^k$ coefficient of $(1-\lambda t^m)^{d-s}\lambda t^m\Hilb(H_\mideal^0(\field_{s-1}[\Delta]), \lambda, t)$ is given by
\[
\sum_{r=0}^i(-1)^r{d-s\choose r}{s-1\choose i-r-1}\beta_{i-r-2}(\Delta)^{mi-k}.
\]
Now we sum over all values of $s$, setting $f(s)={d-s\choose r}{s-1\choose i-r-1}$:
\begin{align*}
\sum_{s=1}^d\sum_{r=0}^i(-1)^r{d-s\choose r}{s-1\choose i-r-1}\beta_{i-r-2}(\Delta)^{mi-k}&=\sum_{r=0}^i\left[(-1)^r\beta_{i-r-2}(\Delta)^{mi-k}\sum_{s=1}^df(s)\right]\\
&={d\choose i}\sum_{r=0}^i(-1)^r\beta_{i-r-2}(\Delta)^{mi-k}\\
&={d\choose i}\sum_{j=0}^{i-1}(-1)^{i-j-1}\beta_{j-1}(\Delta)^{mi-k}.
\end{align*}
Here the second equality follows from the identity $\sum_{r=0}^{d-s}{d-s\choose r}{s-1\choose i-r-1}={d\choose i}$ and the third follows from setting $j=i-r-1$.
\end{proof}
Of particular interest is the $G=\ZZ/2\ZZ$ case, in which many simplifications can be made to the expression in Theorem \ref{SchenzelAnalog}. This results in the following corollary.
\begin{corollary}\label{CsSchenzel}
Let $\Delta$ be a centrally-symmetric Buchsbaum complex with $\Theta=(\theta_1, \ldots, \theta_d)$ an l.s.o.p. for $\Delta$ such that $\theta_i\in A_1^m$ for all $i$ and some fixed $m$. Then
\[
\dim_\field \field(\Delta; \Theta)_i^k = \frac{1}{2}\left(h_i(\Delta)+(-1)^{i+k+mi}{d\choose i}\right)+{d\choose i}\sum_{j=0}^{i-1}(-1)^{i-j-1}\beta_{j-1}(\Delta)^{mi-k}.
\]
\end{corollary}

\subsection{The sigma module}
When only considering a $\ZZ$-grading on $\field[\Delta]$, it is possible to mod $\field(\Delta; \Theta)$ out by an additional submodule in order to get an even tighter bound on the $h$-vector of $\Delta$. In particular, the sigma module can be used to great effect as follows (\cite[Theorem 1.2]{BBMD}).

\begin{theorem}[Murai--Novik--Yoshida]\label{BBMD}
	Let $\Delta$ be a Buchsbaum simplicial complex of dimension $d-1$ and let $\Theta$ be an l.s.o.p. for $\Delta$. Then
	\[
	\dim_\field \left(\field[\Delta]/\Sigma(\Theta; \field[\Delta])\right)_i = h_i(\Delta)+{d\choose i}\sum_{j=0}^i(-1)^{i-j-1}\beta_{j-1}(\Delta).
	\]
\end{theorem}
Of course, we would like for analogous statements to hold for the $(\ZZ\times G)$-graded Hilbert series of $\field[\Delta]/\Sigma(\Theta; \field[\Delta])$. In order to even consider this series, we must first verify that this module is in fact $(\ZZ\times G)$-graded by establishing that the sigma module is fixed by the action of $G$. This is accomplished by the following lemma, whose proof is nearly immediate and has been omitted.

\begin{lemma}
	Let $\Delta$ be a Buchsbaum simplicial complex with a free action by $G$. If $\Theta=(\theta_1, \ldots, \theta_d)$ is an l.s.o.p. for $\Delta$ and each $\theta_i$ is homogeneous with respect to the $(\ZZ\times G)$-grading of $A$, then the sigma module $\Sigma (\Theta; \field[\Delta])$ is fixed by the action of $G$.
\end{lemma}

Fortunately, the proof of our needed aspects of \cite[Theorem 2.3]{BBMD} goes through essentially verbatim in the $(\ZZ\times G)$-graded case, yielding the following proposition (though the stated isomorphism is not necessarily a $\field[G]$-module isomorphism, the graded pieces still go where we want them to go).

\begin{proposition}\label{sigmaModuleProp}Let $\Delta$ be a $(d-1)$-dimensional Buchsbaum simplicial complex admitting a free action by $G$ and let $\Theta=(\theta_1, \ldots, \theta_d)$ be an l.s.o.p. for $\Delta$ with $\theta_i\in A_1^m$ for all $i$. Then
	\[
	\Sigma(\Theta;\field[\Delta])/\Theta\field[\Delta]\cong\bigoplus_{i=0}^{d-1}\left(\bigoplus_{{d\choose i}}H_\mideal^{i}(\field[\Delta])[-i, -im]\right)
	\]
	as vector spaces, and this isomorphism respects the $(\ZZ\times G)$-grading. In particular,
	\[
\dim_{\field}\left(\Sigma(\Theta; \field[\Delta])/\Theta\field[\Delta]\right)_i^k={d\choose i}\beta_{i-1}(\Delta)^{mi-k}.
	\]

\end{proposition}
This immediately establishes the following extension of Theorem \ref{SchenzelAnalog}:
\begin{theorem}\label{SigmaModuleThm}Let $\Delta$ be a $(d-1)$-dimensional Buchsbaum simplicial complex admitting a free action by $G$ and let $\Theta$ be a $G$-homogeneous l.s.o.p. for $\field[\Delta]$ such that $\theta_i\in\field[\Delta]_1^m$ for all $i$. Then
\begin{align*}
\Hilb(\field[\Delta]/\Sigma(\Theta;\field[\Delta]), \lambda, t)=&\sum_{i=0}^d\left[(-1)^i{d\choose i}t^{mi}+\left(\frac{1}{p}\sum_{k=0}^{p-1}t^k\right)\left(h_i(\Delta)+(-1)^{i+1}{d\choose i}\right)\right]\lambda^i\\
&+\sum_{i=0}^d{d\choose i}\lambda^i\sum_{j=0}^{i}(-1)^{i-j-1}\left(\sum_{k=0}^{p-1}t^k\beta_{j-1}(\Delta)^{mi-k} \right).
\end{align*}
\end{theorem}

\subsection{Inequalities}\label{Inequalities}
Theorem \ref{BBMD} implies the following inequality bounding the $h$-vector of any Buchsbaum simplicial complex $\Delta$:
\begin{equation}\label{buchsbaumInequality}
h_i(\Delta)\ge {d\choose i}\sum_{j=0}^i(-1)^{i-j}\beta_{j-1}(\Delta).
\end{equation}
For Cohen-Macaulay complexes admitting a very free action by $G$, bounds on the $h$-vector may be obtained through Theorem \ref{StanleyCMVeryFree}.
In this section we present some similar inequalities by exploiting the $(\ZZ\times G)$-graded Hilbert series of Buchsbaum complexes admitting free group actions by $G$. To begin, setting $m=0$ and examining the $\lambda^it^{k}$ coefficients in Theorem \ref{SigmaModuleThm} immediately allows us to bound the $h$-vector of $\Delta$ by
\begin{equation}\label{zeropart}
	h_i(\Delta)\ge(p-1)(-1)^{i+1}{d\choose i}+p{d\choose i}\sum_{j=0}^i(-1)^{i-j}\beta_{j-1}(\Delta)^0
\end{equation}
and
\begin{equation}\label{nonzeropart}
	h_i(\Delta)\ge(-1)^{i}{d\choose i}+p{d\choose i}\sum_{j=0}^i(-1)^{i-j}\beta_{j-1}(\Delta)^k
\end{equation}
for $k\not\equiv 0 \mod p$.

\begin{remark}
	Note that the $\beta_{j-1}(\Delta)^0$ values have been singled out in these inequalities and that they appear to have a separate influence from any particular $\beta_{j-1}(\Delta)^k$ with $k\not\equiv 0 \mod p$. This will be a recurring theme throughout this section, and we will see that this $G$-invariant component of $\widetilde{H}^i(\Delta)$ has a special significance (see equation (\ref{homologyOfQuotient})).
\end{remark}
	The inequalities above provide an immediate ``permutable'' version of (\ref{buchsbaumInequality}) that is worth recognizing.
\begin{corollary}\label{multiset}Let $\Delta$ be a $(d-1)$-dimensional Buchsbaum simplicial complex admitting a free action by the group $\ZZ/p\ZZ$. If $\mathcal{M}$ is a multiset of size $p-1$ on $\{1, \ldots, p-1\}$, then
	\[
	h_i(\Delta)\ge{d\choose i}\sum_{j=0}^i(-1)^{i-j}\beta_{j-1}(\Delta)^0+{d\choose i}\sum_{m\in\mathcal{M}}\sum_{j=0}^i(-1)^{i-j}\beta_{j-1}(\Delta)^m.
	\]
\end{corollary}
\begin{proof}
For each $m\in \mathcal{M}$, consider the corresponding inequality (\ref{nonzeropart}). Add all of the inequalities together along with (\ref{zeropart}), then divide by $p$.
\end{proof}
Note that by taking $\mathcal{M}=[p-1]$, we obtain the original bound (\ref{buchsbaumInequality}). On the other hand, a direct examination of inequalities (\ref{zeropart}) and (\ref{nonzeropart}) proves the following extension of Theorem \ref{StanleyCMVeryFree}.

\begin{corollary}Let $\Delta$ be a $(d-1)$-dimensional Buchsbaum simplicial complex admitting a free action by $G$. If $\beta_i(\Delta)=0$ for $i\le j$, then the inequalities of Theorem \ref{StanleyCMVeryFree} hold for $h_i(\Delta)$ when $i\le j+2$.
\end{corollary}

In fact, if $\Delta$ is also simply connected, then one of these bounds may be pushed to even higher indices by using some tools from algebraic topology.

\begin{proposition}Let $\Delta$ be a simply-connected Buchsbaum simplicial complex admitting a free action by $G$. If $H_i(\Delta)=0$ for $1\le i\le r$, then 
	\[
	h_i(\Delta)\ge (-1)^i{d\choose i}
	\]
	for $1\le i\le 2r+2$.
\end{proposition}

\begin{proof}Identify $\Delta$ with its geometric realization $|\Delta|$, and set $\Gamma= |\Delta|/G$. Since $G$ acts freely on $\Delta$, the projection $\pi:|\Delta|\to \Gamma$ is a covering space map. By \cite[Proposition 3G.1]{Hatcher},
	\begin{equation}\label{homologyOfQuotient}
	H^i(|\Delta|)^0=H^i(|\Delta|)^G\cong H^i(\Gamma).
	\end{equation}
	If $b$ is a point in $\Gamma$ and $\epsilon$ is the space containing the set of $p$ discrete points $\pi^{-1}(b)$, then $\epsilon\to|\Delta|\to\Gamma$ is a fiber bundle ($|\Delta|$ is a $p$-sheeted covering space for $\Gamma$). Since $\epsilon$ consists of $p$ points,  the long exact sequence of homotopy groups associated to the fiber bundle (see \cite[Theorem 4.41]{Hatcher}) splits into isomorphisms
	\[
	\pi_i(|\Delta|)\cong\pi_i(\Gamma)
	\]
	for $i>1$. Since $\beta_i(\Delta)=0$ for $1\le i\le r$, combining the relative Hurewicz theorem (\cite[Theorem 1]{KK-Hurewicz}) with the isomorphisms above yields
	\[
	H_i(|\Delta|; \QQ)\cong \pi_i(|\Delta|)\otimes \QQ\cong\pi_i(\Gamma)\otimes\QQ\cong H_i(\Gamma;\QQ)
	\]
	for $1\le i\le 2r$. By (\ref{homologyOfQuotient}), this implies that $H^i(|\Delta|)=H^i(|\Delta|)^0$ for $1\le i\le 2r$. In particular, $H^i(\Delta)^k=0$ for $k\not\equiv 0\mod p$. Now apply (\ref{nonzeropart}).
\end{proof}

As indicated by the previous proof, it may be worth examining the properties of the quotient $\Gamma:=\Delta/G$. To that end, we now present the following proposition.

\begin{proposition}Let $\Delta$ be a Buchsbaum simplicial complex admitting a free group action by $G$, and let $\Gamma$ be the quotient $\Delta/G$. If $\Gamma$ is a simplicial complex (or, more generally, a simplicial poset), then
\[
h_i(\Gamma)\ge (-1)^i{d\choose i}+{d\choose i}\sum_{j=0}^i(-1)^{i-j}\beta_{j-1}(\Delta)^k
\]
for $k\not\equiv 0 \mod p$.
\end{proposition}

\begin{proof}A straightforward calculation shows that
\begin{equation}\label{hVectorOfQuotient}
h_i(\Delta)=(p-1)(-1)^{i+1}{d\choose i}+p h_i(\Gamma).
\end{equation}
Now combine (\ref{nonzeropart}) with (\ref{hVectorOfQuotient}) and divide by $p$.
\end{proof}
As examples, the hypotheses above are satisfied when $\Delta$ is the order complex of a poset (\cite{Garsia}), or, more generally, when $\Delta$ is balanced and the action of $G$ is color-preserving (\cite{Reiner}).

\subsection{Dehn-Sommerville relations}
Here we present a new results akin to Theorem 1.4, relating the symmetric values of the dimensions $\dim_\field \left(\field[\Delta]/\Sigma(\Theta;\field[\Delta])\right)_i^j$. First, let $\tilde{\chi}(\Delta)$ denote the reduced Euler characteristic of a simplicial complex $\Delta$, defined by
\[
\tilde{\chi}(\Delta)=\sum_{j=-1}^{d-1}(-1)^j\beta_j(\Delta),
\]
where $d-1$ is the dimension of $\Delta$. We will also denote by $\chi(\Delta):=\tilde{\chi}(\Delta)+1$ the (unreduced) Euler characteristic of $\Delta$. If $\Delta$ is an orientable homology manifold, then Klee's formula 
\begin{equation}\label{KleesFormula}
h_{d-i}(\Delta)-h_i(\Delta)=(-1)^i{d\choose i}((-1)^{d-1}\tilde{\chi}(\Delta)-1)
\end{equation}
(found in \cite{Klee}) provides a relation between the symmetric entries in the $h$-vector of $\Delta$ in terms of $\tilde{\chi}(\Delta)$, abstracting the Dehn-Sommerville relations for simplicial polytopes. We will first examine the symmetric values of the dimensions of $\dim_\field \left(\field[\Delta]/\Sigma(\Theta;\field[\Delta])\right)_i^0$, in the case that $\Theta\subset A_1^0$.
\begin{theorem}\label{DSTheorem}
	Let $\Delta$ be a $(d-1)$-dimensional triangulation of a connected orientable manifold admitting a free group action by $G$. If $\Theta\subset A_1^0$ and $\widetilde{H}^{d-1}(\Delta)=\widetilde{H}^{d-1}(\Delta)^0$, then
	\begin{equation}\label{DS}
	\dim_\field \left(\field[\Delta]/\Sigma(\Theta;\field[\Delta])\right)_i^0=\dim_\field \left(\field[\Delta]/\Sigma(\Theta;\field[\Delta])\right)_{d-i}^0.
	\end{equation}
\end{theorem}

\begin{proof}Applying Theorem \ref{SigmaModuleThm}, we can write
	\[
	\dim_\field \left(\field[\Delta]/\Sigma(\Theta;\field[\Delta])\right)_i^0=\frac{h_i(\Delta)}{p}+\frac{(-1)^i(p-1)}{p}{d\choose i}+{d\choose i}\sum_{j=0}^i(-1)^{i-j-1}\beta_{j-1}(\Delta)^0.
	\]
	Using (\ref{KleesFormula}),
	\begin{align*}
	\frac{h_i(\Delta)}{p}-\frac{h_{d-i}(\Delta)}{p}&=\frac{(-1)^{i+1}}{p}{d\choose i}((-1)^{d-1}\tilde{\chi}(\Delta)-1)\\
	&=\frac{(-1)^{d-i}}{p}{d\choose i}\tilde{\chi}(\Delta)+\frac{(-1)^i}{p}{d\choose i}\\
	&=\frac{(-1)^{d-i}}{p}{d\choose i}\chi(\Delta)+\frac{(-1)^{d-i-1}}{p}{d\choose i}+\frac{(-1)^i}{p}{d\choose i}.
	\end{align*}
	Hence,
	\[
	\left[\frac{h_i(\Delta)}{p}+\frac{(-1)^i(p-1)}{p}{d\choose i}\right]-\left[\frac{h_{d-i}(\Delta)}{p}+\frac{(-1)^{d-i}(p-1)}{p}{d\choose i}\right]
	\]
	\begin{equation}\label{hVectorDifference}
	=\frac{(-1)^{d-i}}{p}{d\choose i}\chi(\Delta)+(-1)^i{d\choose i}+(-1)^{d-i-1}{d\choose i}.
	\end{equation}
	By (\ref{homologyOfQuotient}), $\beta_{j-1}(\Delta)^0=\beta_{j-1}(\Gamma)$, where $\Gamma$ is the quotient space $|\Delta|/G$. Since $\widetilde{H}^{d-1}(\Delta)=\widetilde{H}^{d-1}(\Delta)^0$ by assumption, $\Gamma$ is itself an orientable manifold. Hence, Poincar\'{e} duality applies and $\beta_{j-1}(\Delta)^0=\beta_{d-j}(\Delta)^0$ for $j>1$ while $\beta_0(\Delta)^0=\beta_{d-1}(\Delta)^0-1$. Then
	\begin{align*}
	{d\choose i}\sum_{j=0}^i(-1)^{i-j-1}\beta_{j-1}(\Delta)^0&=(-1)^{i+1}{d\choose i}+{d\choose i}\sum_{j=0}^i(-1)^{i-j-1}\beta_{d-j}(\Delta)^0\\
	&=(-1)^{i+1}{d\choose i}+{d\choose i}\sum_{\ell=i+1}^{d}(-1)^{d-i-\ell}\beta_{\ell -1}(\Delta)^0.
	\end{align*}
	Thus,
	\begin{align*}
	{d\choose i}\sum_{j=0}^i(-1)^{i-j-1}\beta_{j-1}(\Delta)^0&-{d\choose i}\sum_{j=0}^{d-i}(-1)^{d-i-j-1}\beta_{j-1}(\Delta)^0\\
	&=(-1)^{i+1}{d\choose i}+{d\choose i}\sum_{j=0}^d(-1)^{d-i-j}\beta_{j-1}(\Delta)^0\\
	&=(-1)^{i+1}{d\choose i}+(-1)^{d-i-1}{d\choose i}\tilde{\chi}(\Gamma)
	\end{align*}
	Now note that $\Delta$ is a $p$-sheeted covering space for $\Gamma$. This implies that $\chi(\Delta)=p\chi(\Gamma)$ (see \cite[Section 2.2]{Hatcher}). Hence, we can re-write the difference as
	\begin{align*}
	(-1)^{i+1}{d\choose i}+(-1)^{d-i-1}{d\choose i}\tilde{\chi}(\Gamma)&=(-1)^{i+1}{d\choose i}+(-1)^{d-i-1}{d\choose i}\chi(\Gamma)+(-1)^{d-i}{d\choose i}\\
	&=(-1)^{i+1}{d\choose i}+\frac{(-1)^{d-i-1}}{p}{d\choose i}\chi(\Delta)+(-1)^{d-i}{d\choose i}.
	\end{align*}
	Adding this difference to (\ref{hVectorDifference}) now completes the proof.
\end{proof}

\begin{remark}
	In the case that $G=\ZZ/p\ZZ$ with $p$ odd, the hypothesis $\widetilde{H}^{d-1}(\Delta)=\widetilde{H}^{d-1}(\Delta)^0$ of Theorem \ref{DSTheorem} is always true. Indeed, the fundamental class of $\Delta$ generating $\widetilde{H}^{d-1}(\Delta)$ is of the form
	\[
	z = \sum_{\widehat{\sigma}\in C^{d-1}(\Delta)}a_\sigma \widehat{\sigma}
	\]
	with $a_\sigma=\pm 1$ for all $\sigma$. If $h\in G$ acts on this class, then it can only flip the sign of some coefficients; it cannot introduce a factor of $\zeta^i$ across all coefficients for some non-trivial $i$.
\end{remark}

In fact, the theorem above can be greatly strengthened by appealing to algebraic properties of $\field[\Delta]/\Sigma(\Theta;\field[\Delta])$. First, we will need the following theorem (see \cite[Theorem 1.4]{NS-Gorenstein} or \cite[Remark 3.8]{BBMD}).
\begin{theorem}\label{Gorenstein}
	Let $\Delta$ be a triangulation of a $(d-1)$-dimensional connected manifold that is orientable over $\field$, and let $\Theta$ be a linear system of parameters for $\field[\Delta]$. Then $\field[\Delta]/\Sigma(\Theta;\field[\Delta])$ is an Artinian Gorenstein $\field$-algebra.
\end{theorem}
With this fact, our next theorem quickly follows.
\begin{theorem}\label{DSTheorem2}
	Let $\Delta$ be a triangulation of a $(d-1)$-dimensional connected manifold that is orientable over $\field$ and admits a free group action by $G$, and let $\Theta$ be a $(\ZZ\times G)$-homogeneous linear system of parameters for $\field[\Delta]$. If $s$ is such that $\field[\Delta]/\Sigma(\Theta;\field[\Delta])_d$ in concentrated in $(\ZZ\times G)$-degree $(d, s)$, then
	\[
	\dim_\field \left(\field[\Delta]/\Sigma(\Theta;\field[\Delta])\right)_i^j = \dim_\field \left(\field[\Delta]/\Sigma(\Theta;\field[\Delta])\right)_{d-i}^{s-j}
	\]
	for $i=0, \ldots, d$.
\end{theorem}
Before starting the proof of this theorem, note that $\dim_{\field}\field[\Delta]/\Sigma(\Theta;\field[\Delta])_d=1$ because $\field[\Delta]/\Sigma(\Theta;\field[\Delta])$ is Gorenstein and hence such an $s$ always exists.
\begin{proof}
	Since $\field[\Delta]/\Sigma(\Theta;\field[\Delta])$ is Gorenstein, the product map
	\[
	\field[\Delta]/\Sigma(\Theta;\field[\Delta])_i \times \field[\Delta]/\Sigma(\Theta;\field[\Delta])_{d-i} \to \field[\Delta]/\Sigma(\Theta;\field[\Delta])_d
	\]
	is a perfect pairing (see \cite[Theorem 2.79]{LefschetzProperties}). Furthermore, the finely-graded product maps
	\[
	\field[\Delta]/\Sigma(\Theta;\field[\Delta])_i^j \times \field[\Delta]/\Sigma(\Theta;\field[\Delta])_{d-i}^k \to \field[\Delta]/\Sigma(\Theta;\field[\Delta])_d^{j+k}
	\]
	only land in a non-zero component of $\field[\Delta]/\Sigma(\Theta;\field[\Delta])_d$ when $j+k\equiv s \pmod p$. That is,
	\[
	\field[\Delta]/\Sigma(\Theta;\field[\Delta])_i^j \times \field[\Delta]/\Sigma(\Theta;\field[\Delta])_{d-i}^{s-j} \to \field[\Delta]/\Sigma(\Theta;\field[\Delta])_d^{s}
	\]
	is a perfect pairing for each $i, j$ and hence
	\[
	\field[\Delta]/\Sigma(\Theta;\field[\Delta])_i^j \cong \field[\Delta]/\Sigma(\Theta;\field[\Delta])_{d-i}^{s-j}
	\]
	as vector spaces over $\field$.
\end{proof}
By comparing \ref{DSTheorem} with \ref{DSTheorem2}, we obtain the following corollary.
\begin{corollary}
	Let $\Delta$ be a $(d-1)$-dimensional triangulation of a connected orientable manifold admitting a free group action by $G$. If $\Theta\subset A_1^0$ and $\widetilde{H}^{d-1}(\Delta)=\widetilde{H}^{d-1}(\Delta)^0$, then $\field[\Delta]/\Sigma(\Theta;\field[\Delta])_d$ in concentrated in $(\ZZ\times G)$-degree $(d, 0)$ and
	\[
	\dim_\field \left(\field[\Delta]/\Sigma(\Theta;\field[\Delta])\right)_i^j = \dim_\field \left(\field[\Delta]/\Sigma(\Theta;\field[\Delta])\right)_{d-i}^{p-j}
	\]
	for $i=0, \ldots, d$ and $j=0, \ldots, p-1$.
\end{corollary}

\section{Comments}\label{comments}
It was shown by Duval in \cite{Duval} that a version of of Theorem \ref{Reisner} also holds when considering the face ring of a simplicial poset. Our main result then begs the following question.
\begin{question}
	Does Theorem \ref{ReisnerAnalog} extend to simplicial posets admitting free group actions?
\end{question}
An answer to this question will likely require a careful examination of a $(\ZZ\times G)$-grading imposed on the face ring of a simplicial poset, an object that at times is considerably more complicated than the Stanley--Reisner ring.

Stanley was able to examine some implications of equality being attained in his inequalities of Theorem \ref{StanleyCMVeryFree} for the $G=\ZZ/2\ZZ$ case (see \cite[Proposition III.8.2]{St-96}). In light of these results, our next question naturally arises.
\begin{question}
	If equality is attained in either inequality (\ref{zeropart}) or (\ref{nonzeropart}), what can be said about the other entries in the $h$-vector of $\Delta$?
\end{question}

Lastly, this paper would feel incomplete without some mention of the $g$-conjecture. To this end, let $\Delta$ be a $(d-1)$-dimensional orientable $\field$-homology manifold and denote
\[
h_i''(\Delta)=\dim_\field \left(\field[\Delta]/\Sigma(\Theta;\field[\Delta])\right)_i=h_i(\Delta)+{d\choose i}\sum_{j=0}^i(-1)^{i-j-1}\beta_{j-1}(\Delta).
\]
Kalai's manifold $g$-conjecture asserts in part that $h_0''(\Delta)\le h_1''(\Delta)\le \cdots \le h_{\lfloor d/2\rfloor}''(\Delta)$ (in the case that $\Delta$ is a sphere, this is the same as the corresponding part of the usual $g$-conjecture). There is much evidence in favor of this statement, and it has been shown to be true in some special cases (see \cite[Theorem 5.3]{Tight}, \cite[Theorem 1.5]{FaceNumbers}, and \cite[Theorem 1.6]{NS-Gorenstein}). The inequality is sometimes demonstrated by exhibiting a Lefschetz element $\omega\in A_1$ such that the multiplication map
\[
\cdot\omega:\field[\Delta]/\Sigma(\Theta;\field[\Delta])_{i-1}\to\field[\Delta]/\Sigma(\Theta;\field[\Delta])_{i}
\]
is an injection for $1\le i\le \lfloor d/2\rfloor$. As has been our theme, we would like to show that similar inequalities hold under the finer grading. Indeed, if such an $\omega$ can in fact be chosen to be an element of $A_1^m$ for some $m$, then we immediately have the inequalities
\[
\dim_\field \left(\field[\Delta]/\Sigma(\Theta;\field[\Delta])\right)_{i-1}^j\le \dim_\field \left(\field[\Delta]/\Sigma(\Theta;\field[\Delta])\right)_i^{j+m}
\]
for any choice of $j$. Although the Lefschetz elements that have been found thus far cannot be specified to reside in some fixed homogeneous $(\ZZ\times G)$-degree of $A_1$, the strength of the finer grading on $\field[\Delta]/\Sigma(\Theta;\field[\Delta])$ prompts the following conjecture.

\begin{conjecture}Let $\Delta$ be an orientable $\field$-homology manifold admitting a free group action by $\ZZ/p\ZZ$. Then there exists $m$ such that
\[
\dim_\field \left(\field[\Delta]/\Sigma(\Theta;\field[\Delta])\right)_{i-1}^j\le \dim_\field \left(\field[\Delta]/\Sigma(\Theta;\field[\Delta])\right)_i^{j+m}
\]
for $1\le i\le \lfloor d/2\rfloor$ and $0\le j\le p-1$.
\end{conjecture}
\section*{Acknowledgements}The author would like to thank Martina Juhnke-Kubitzke, Isabella Novik, Travis Scholl, and an anonymous referee for many helpful comments and conversations.

\bibliographystyle{alpha}
\bibliography{GroupActions-biblio}
\end{document}